\documentclass[leqno,onefignum,onetabnum]{siamltex1213}
\usepackage{amsfonts,amsmath,amssymb,mathrsfs}
\usepackage[commentmarkup=uwave]{changes}
\usepackage{booktabs,ctable,threeparttable,multirow}
\usepackage{graphicx,boxedminipage,appendix}
\usepackage[caption=false]{subfig}
\usepackage{algorithm,algorithmic}
\usepackage{lineno,soul,color,xcolor}

\usepackage{amsthm,bm,enumerate,hyperref}
\newtheorem{theoremmy}{Theorem}[section]

\newtheorem{exper}[theoremmy]{Example}
\numberwithin{equation}{section}

\newcommand{\rmnum}[1]{\romannumeral #1}

\newcommand{\new}[0]{\mathrm{new}}

\newcommand{\N}[0]{\mathcal{N}}
\newcommand{\X}[0]{\mathcal{X}}

\newcommand{\Z}[0]{\mathcal{Z}}
\newcommand{\UU}[0]{\mathcal{U}}
\newcommand{\V}[0]{\mathcal{V}}

\newcommand{\bsmallmatrix}[1]{\begin{bmatrix}\begin{smallmatrix}
#1\end{smallmatrix}\end{bmatrix}}

\title{Two harmonic Jacobi--Davidson methods for computing a partial generalized
singular value decomposition of a large
matrix pair\thanks{Supported by the National Natural Science
Foundation of China (No. 12171273).}}

\author{
Jinzhi Huang\thanks{School of Mathematical Sciences, Soochow University,
215006 Suzhou, China
(\url{jzhuang21@suda.edu.cn})}.
\and
Zhongxiao Jia\thanks{Corresponding author. Department of Mathematical Sciences,
Tsinghua University, 100084 Beijing, China
(\url{jiazx@tsinghua.edu.cn}). The two authors contributed equally to this work.}
}

\begin{document}
\maketitle
\begin{abstract}
Two harmonic extraction based Jacobi--Davidson (JD) type
algorithms are proposed to compute a partial generalized singular value
decomposition (GSVD) of a large regular matrix pair.
They are called cross product-free (CPF) and inverse-free (IF) harmonic
JDGSVD algorithms, abbreviated as CPF-HJDGSVD and IF-HJDGSVD, respectively.
Compared with the standard extraction based JDGSVD algorithm,
the harmonic extraction based algorithms converge more regularly
and suit better for computing
GSVD components corresponding to interior generalized singular values.
Thick-restart CPF-HJDGSVD and IF-HJDGSVD algorithms
with some deflation and purgation techniques are developed to compute
more than one GSVD components.
Numerical experiments confirm the superiority of CPF-HJDGSVD and IF-HJDGSVD to
the standard extraction based JDGSVD algorithm.
\end{abstract}



\pagestyle{myheadings}
\thispagestyle{plain}
\markboth{HARMONIC JDGSVD METHODS FOR GSVD COMPUTATIONS}
{JINZHI HUANG AND ZHONGXIAO JIA}

\section{Introduction}\label{sec:1}
For a pair of large and possibly sparse matrices
$A\in\mathbb{R}^{m\times n}$ and $B\in\mathbb{R}^{p\times n}$,
the matrix pair $(A,B)$ is called regular if $\N(A)\cap\N(B)=\{\bm{0}\}$,
i.e., $\mathrm{rank}\left(\bsmallmatrix{A\\B}\right)=n$,
where $\N(A)$ and $\N(B)$ denote the null spaces of $A$ and $B$,
respectively. The generalized singular value decomposition (GSVD) of $(A,B)$
was introduced by Van Loan \cite{van1976generalizing}
and developed by Paige and Saunders \cite{paige1981towards}. Since then,
GSVD has become a standard matrix decomposition and has
been widely used \cite{betcke2008generalized,bjorck1996numerical,
chu1987singular,golub2012matrix,hansen1998rank,kaagstrom1984generalized}.
Let $q_1=\dim(\N(A))$, $q_2=\dim(\N(B))$ and $l_1=\dim(\N(A^T))$,
$l_2=\dim(\N(B^T))$, where the superscript $T$ denotes the transpose.
Then the GSVD of $(A,B)$ is
\begin{equation}\label{Gsvd}
\left\{\begin{aligned}
&U^TAX=\Sigma_A=\diag\{C,\mathbf{0}_{l_1, q_1},I_{q_2}\}, \\
&V^TBX=\Sigma_B=\diag\{S,I_{q_1},\mathbf{0}_{l_2, q_2}\},
\end{aligned}\right.
\end{equation}
where $X=[X_q,X_{q_1},X_{q_2}]$ is nonsingular,
$U=[U_q,U_{l_1},U_{q_2}]$ and $V=[V_q,V_{q_1},V_{l_2}]$  are orthogonal, and
the diagonal matrices
$C=\diag\{\alpha_1,\dots,\alpha_q\}$ and
$S=\diag\{\beta_1,\dots,\beta_q\}$ satisfy
\begin{equation*}
  0<\alpha_i,\beta_i<1 \quad\mbox{and}\quad
 \alpha_i^2+\beta_i^2=1,  \quad i=1,\dots,q
\end{equation*}
with $q=n-q_1-q_2$.
Here, $\mathbf{0}_{l_i,q_i}$ and $I_{q_i}, i=1,2,$
are the $l_i\times q_i$
zero matrices and identity matrices of order $q_i$, respectively;
see \cite{paige1981towards}. The GSVD part in \eqref{Gsvd} that
corresponds to $\alpha_i$ and $\beta_i$ can be written as
\begin{equation}\label{gsvd}
  \left\{
  \begin{aligned}
  Ax_i&=\alpha_i u_i,\\
  Bx_i&=\beta_i v_i,\\
  \beta_i A^Tu_i&=\alpha_iB^Tv_i,
  \end{aligned}
  \right.
  \qquad
  i=1,\dots,q,
\end{equation}
where $x_i$ is the $i$th column of $X_q$ and
the unit-length vectors $u_i$ and $v_i$ are the $i$th columns of
$U_q$ and $V_q$, respectively.
The quintuples $(\alpha_i,\beta_i,u_i,v_i,x_i)$, $i=1,\dots,q$
are called nontrivial GSVD components of $(A,B)$.
Particularly, the numbers $\sigma_i=\frac{\alpha_i}{\beta_i}$ or the
pairs $(\alpha_i,\beta_i)$ are called the nontrivial
generalized singular values, and $u_i,v_i$ and $x_i$ are the
corresponding left and right generalized singular vectors, respectively, $i=1,\dots,q$.

For a given target $\tau>0$, we assume that all the nontrivial
generalized singular values of $(A,B)$ are labeled
by their distances from $\tau$:
\begin{equation}\label{order}
  |\sigma_1-\tau|\leq\dots\leq|\sigma_{\ell}-\tau|<
  |\sigma_{\ell+1}-\tau|\leq\dots\leq|\sigma_q-\tau|.
\end{equation}
We are interested in computing the  GSVD components
$(\alpha_i,\beta_i,u_i,v_i,x_i)$ corresponding to the $\ell$ nontrivial
generalized singular values $\sigma_i$ of $(A,B)$ closest to $\tau$.
If $\tau$ is inside the nontrivial generalized singular spectrum of $(A,B)$,
then $(\alpha_i,\beta_i,u_i,v_i,x_i)$, $i=1,\dots,\ell$ are called
interior GSVD components of $(A,B)$; otherwise, they are called
the extreme, i.e., largest or smallest, ones.
A large number of GSVD components, some of which are interior
ones \cite{chuiwang2010,coifman2006,coifman2005},
are required in a variety of applications.
Throughout this paper, we assume that $\tau$ is not equal to any
generalized singular value of $(A,B)$.

Zha \cite{zha1996computing} proposes a
joint bidiagonalization (JBD) method
to compute extreme GSVD components of the large matrix pair $(A,B)$.
The method is based on a JBD process that successively reduces
$(A,B)$ to a sequence of upper bidiagonal pairs, from which
approximate GSVD components are computed. Kilmer, Hansen and Espanol
\cite{kilmer2007} have adapted
the JBD process to the linear discrete ill-posed problem with general-form
regularization and developed a JBD process that reduces
$(A,B)$ to lower-upper bidiagonal forms. Jia and Yang~\cite{JiaYang2020}
have developed a new JBD process based iterative algorithm for the ill-posed
problem and considered the convergence of extreme generalized singular values.
In the GSVD computation and the solution of discrete ill-posed problem,
one needs to solve an $(m+p)\times n$ least squares
problem with the coefficient matrix $[A^T,B^T]^T$ at each step of the JBD process.
Jia and Li \cite{jia2021joint} have recently considered the JBD process
in finite precision and proposed a partial
reorthogonalization strategy to maintain numerical semi-orthogonality among the
generated basis vectors so as to avoid ghost approximate GSVD components,
where the semi-orthogonality means that two unit-length vectors
are numerically orthogonal to the level of $\epsilon_{\rm mach}^{1/2}$
with $\epsilon_{\rm mach}$ being the machine precision.

Hochstenbach \cite{hochstenbach2009jacobi} presents a Jacobi--Davidson
(JD) GSVD (JDGSVD) method to compute a number of
interior GSVD components of $(A,B)$ with $B$ of full column rank,
where, at each step, an $(m+n)$-dimensional linear system,
i.e., the correction equation, needs to be solved iteratively with
low or modest accuracy; see \cite{huang2019inner,huang2020cross,
jia2014inner,jia2015harmonic}.
The lower $n$-dimensional and upper $m$-dimensional parts of the approximate
solution are used to expand the right and one of the left searching
subspaces, respectively.
The JDGSVD method formulates the GSVD of $(A,B)$ as the
equivalent generalized eigendecomposition of the augmented matrix pair
$\left(\bsmallmatrix{&A\\A^T&},\bsmallmatrix{I&\\&B^TB}\right)$
for $B$ of full column rank, computes the relevant eigenpairs,
and recovers the approximate GSVD components from the converged eigenpairs.
The authors \cite{huang2020choices} have shown that the error of the
computed eigenvector is bounded by the size of the perturbations
times a multiple $\kappa(B^TB)=\kappa^2(B)$, where
$\kappa(B)=\sigma_{\max}(B)/\sigma_{\min}(B)$ denotes the $2$-norm
condition number of $B$ with $\sigma_{\max}(B)$ and $\sigma_{\min}(B)$
being the largest and smallest singular values of $B$, respectively.
Consequently, with an ill-conditioned $B$, the computed GSVD components
may have very poor accuracy, which has been numerically confirmed \cite{huang2020choices}.
The results in \cite{huang2020choices} show that if $B$ is ill conditioned
but $A$ has full column rank and is well conditioned then the JDGSVD method
can be applied to the matrix pair $(\bsmallmatrix{&B\\B^T&},\bsmallmatrix{I&\\&A^TA})$
and computes the corresponding approximate GSVD components with high accuracy. Note
that the two formulations require that $B$ and $A$ be rectangular or square,
respectively. We should also realize that a reliable estimation of the condition
numbers of $A$ and $B$ may be costly, so that it may be difficult to
choose a proper formulation in applications.

Zwaan and Hochstenbach \cite{zwaan2017generalized} present a generalized Davidson
(GDGSVD) method and a multidirectional (MDGSVD) method to compute an {\em extreme}
partial GSVD of $(A,B)$.
These two methods involve no cross product matrices $A^TA$ and
$B^TB$ or matrix-matrix products, and
they apply the standard extraction approach, i.e.,
the Rayleigh--Ritz method \cite{stewart2001matrix} to $(A,B)$ directly and
compute approximate GSVD components with respect to the given left and right
searching subspaces, where the two left subspaces are formed by premultiplying
the right one with $A$ and $B$, respectively.
At iteration $k$ of the GDGSVD method, the right searching subspace is spanned
by the $k$ residuals of the generalized Davidson method
\cite[Sec.~11.2.4 and Sec.~11.3.6]{bai2000} applied to
the generalized eigenvalue problem of $(A^TA,B^TB)$;
in the MDGSVD method, an inferior search direction is discarded by a truncation
technique, so that the searching subspaces are improved.
Zwaan \cite{zwaan2019} exploits the Kronecker canonical form of a regular matrix pair \cite{stewart90} and shows that the GSVD problem of $(A,B)$ can be formulated
as a certain $(2m+p+n)\times (2m+p+n)$ generalized eigenvalue
problem without involving any cross product or any other matrix-matrix product.
Such formulation currently is mainly of theoretical value since the nontrivial
eigenvalues and eigenvectors of the structured generalized eigenvalue problem
are always complex: the generalized eigenvalues are the conjugate quaternions
$(\sqrt{\sigma_j},-\sqrt{\sigma_j},\mathrm{i}\sqrt{\sigma_j},-\mathrm{i}
\sqrt{\sigma_j})$ with $\mathrm{i}$ the imaginary unit, and
the corresponding right generalized eigenvectors are
\begin{eqnarray*}
&&[u_j^T,x_j^T/\beta_j,\sqrt{\sigma_j}u_j^T,\sqrt{\sigma_j}v_j^T]^T,
\qquad\qquad\hspace{0.12em}
[-u_j^T,-x_j^T/\beta_j,\sqrt{\sigma_j}u_j^T,\sqrt{\sigma_j}v_j^T]^T,
\\
&&[-\mathrm{i}u_j^T,\mathrm{i}x_j^T/\beta_j,
\sqrt{\sigma_j}u_j^T,-\sqrt{\sigma_j}v_j^T]^T,\qquad
 [\mathrm{i}u_j^T,\mathrm{i}x_j^T/\beta_j,-
 \sqrt{\sigma_j}u_j^T,-\sqrt{\sigma_j}v_j^T]^T.
\end{eqnarray*}
Clearly, the size of the generalized eigenvalue problem is much
bigger than that of the GSVD of $(A,B)$.
The conditioning of eigenvalues and eigenvectors of this problem is also unclear.
In the meantime, no structure-preserving algorithm has been
found for such kind of complicated structured generalized eigenvalue problem.
Definitely, it will be extremely difficult and highly challenging to
seek for a numerically stable structure-preserving algorithm for
this problem.

The authors \cite{huang2020cross} have recently proposed
a Cross Product-Free JDGSVD method, referred to as the CPF-JDGSVD method,
to compute several GSVD components of $(A,B)$ corresponding to the generalized
singular values closest to $\tau$. The CPF-JDGSVD method is cross products
$A^TA$ and $B^TB$ free when constructing and expanding right and left
searching subspaces; it premultiplies the right searching subspace
by $A$ and $B$ to construct two left ones separately, and
forms the orthonormal bases of those by computing two thin QR
factorizations, as done in \cite{zwaan2017generalized}.
The resulting projected problem is the GSVD of a small matrix pair
without involving any cross product or matrix-matrix product.
Mathematically, the method implicitly deals with the
equivalent generalized eigenvalue problem of
$(A^TA,B^TB)$ without forming $A^TA$ or $B^TB$ explicitly.
At the subspace expansion stage, an $n$-by-$n$ correction equation is
approximately solved iteratively with low or modest accuracy,
and the approximate solution is used to expand the searching subspaces. Therefore,
the subspace expansion is fundamentally different from that used in
\cite{zwaan2017generalized}, where the dimension $n$ of the correction equations is
no more than half of the dimension $m+n$ of those
in \cite{hochstenbach2009jacobi}.

Just like the standard Rayleigh--Ritz method for
the matrix eigenvalue problem and the singular value decomposition (SVD)
problem, the CPF-JDGSVD method suits better for the computation
of some extreme GSVD components, but it may encounter some
serious difficulties for the computation of interior GSVD components.
Remarkably, adapted from the standard extraction
approach for the eigenvalue
problem and SVD problem to the GSVD computation, an
intrinsic shortcoming of a standard extraction based method
is that it may be hard to pick up
good approximate generalized singular values correctly even if
the searching subspaces are sufficiently good. This potential disadvantage may make
the resulting algorithm expand the subspaces along wrong directions and
converge irregularly, as has been numerically observed in \cite{huang2020cross}.
To this end, inspired by the harmonic extraction based methods that
suit better for computing interior eigenpairs and SVD components
\cite{hochstenbach2004harmonic,hochstenbach2008harmonic,huang2019inner,
jia2002refinedh,jia2010refined,jia2015harmonic,morgan1998harmonic},
we will propose two harmonic extraction based JDGSVD methods
that are particularly suitable for the computation of interior GSVD components.
One method is cross products $A^TA$ and $B^TB$ free, and the
other is inversions $(A^TA)^{-1}$ and $(B^TB)^{-1}$ free. As will be seen,
the derivations of the
two harmonic extraction methods are nontrivial, and they are
subtle adaptations of
the harmonic extraction for matrix eigenvalue and SVD problems.
In the sequel, we will abbreviate Cross Product-Free and
Inverse-Free Harmonic JDGSVD methods as CPF-HJDGSVD and IF-HJDGSVD, respectively.

We first focus on the case $\ell=1$ and propose our harmonic extraction based
JDGSVD type methods. Then by introducing the deflation technique in
\cite{huang2020cross} into the methods, we present the methods to
compute more than one, i.e., $\ell>1$, GSVD components. To be practical,
combining the thick-restart technique in \cite{stath1998} and some
purgation approach, we develop thick-restart CPF-HJDGSVD and IF-HJDGSVD
algorithms to compute the
$\ell$ GSVD components associated with the generalized
singular values of $(A,B)$ closest to $\tau$.

The rest of this paper is organized as follows.
In Section~\ref{sec:2}, we briefly review the
CPF-JDGSVD method proposed in \cite{huang2020cross}.
In Section~\ref{sec:3}, we propose the CPF-HJDGSVD and IF-HJDGSVD methods.
In Section~\ref{sec:5}, we develop thick-restart CPF-HJDGSVD and IF-HJDGSVD
with deflation and purgation to compute $\ell$ GSVD components of $(A,B)$.  In
Section~\ref{sec:6}, we report numerical experiments to
illustrate the performance of CPF-HJDGSVD and IF-HJDGSVD, make
a comparison of them and CPF-JDGSVD, and show the superiority of
the former two to the latter one.
Finally, we conclude the paper in Section~\ref{sec:7}.

Throughout this paper, we denote by $\mathcal{R}(\cdot)$ the
column space of a matrix, and by $\|\cdot\|$ and $\|\cdot\|_1$ the $2$- and
$1$-norms of a matrix or vector, respectively. As in \eqref{Gsvd}, we denote
by $I_i$ and $\bm{0}_{i,j}$ the $i$-by-$i$ identity and $i$-by-$j$ zero matrices,
respectively, with the subscripts $i$ and $j$ dropped whenever
they are clear from the context.

\section{The standard extraction based JDGSVD method}\label{sec:2}
We review the CPF-JDGSVD method in \cite{huang2020cross}
for computing the GSVD component $(\alpha_*,\beta_*,u_*,v_*,x_*): =
(\alpha_1,\beta_1,u_1,v_1,x_1)$ of $(A,B)$.
Assume that a $k$-dimensional right searching subspace
$\X\subset \mathbb{R}^{n}$ is available, from which an approximation
to $x_*$ is extracted. Then we construct
\begin{equation}\label{search}
   \UU=A \X
\qquad\mbox{and}\qquad
 \V=B\X
\end{equation}
as the two left searching subspaces, from which approximations to
$u_*$ and $v_*$ are computed. It is proved in \cite{huang2020cross}
that the distance between $u_*$ and
$\UU$ (resp. $v_*$ and $\V$) is as small as that between $x_*$
and $\X$, provided that $\alpha_*$ (resp. $\beta_*$) is not very small.
In other words, for the extreme and interior GSVD components,
$ \UU$ and $ \V$ constructed by \eqref{search} are as good as $\X$
provided that the desired generalized singular values
$\sigma_*=\frac{\alpha_*}{\beta_*}$  are neither very
small nor very small. It is also proved in \cite{huang2020cross} that
$\UU$ or $\V$ is as accurate as $\X$
for very large or small generalized singular values.

Assume that the columns of $\widetilde X\in\mathbb{R}^{n\times k}$
form an orthonormal basis of $\X$, and compute the thin QR factorizations
of $A\widetilde X$ and $B\widetilde X$:
\begin{equation}\label{qrAXBX}
  A\widetilde  X= \widetilde UR_{A}
  \qquad\mbox{and}\qquad
  B\widetilde  X= \widetilde VR_{B},
\end{equation}
where $\widetilde U\in\mathbb{R}^{m\times k}$ and $\widetilde V\in\mathbb{R}^{p\times k}$
are orthonormal, and $R_{A}\in\mathbb{R}^{k\times k}$ and
$R_{B}\in\mathbb{R}^{k\times k}$ are upper triangular.
Then the columns of $\widetilde U$ and $\widetilde V$ are orthonormal bases
of $\UU$ and $\V$, respectively.
With $\X$, $\UU$, $\V$ and their orthonormal bases available,
we can extract an approximation to the desired GSVD component
$(\alpha_*,\beta_*,u_*,v_*,x_*)$ of $(A,B)$ with respect to them.
The standard extraction approach in \cite{huang2020cross}
seeks for positive pairs
$(\tilde\alpha,\tilde\beta)$ with $\tilde\alpha^2+\tilde\beta^2=1$,
normalized vectors $\tilde u\in\UU$, $\tilde v\in\V$, and
vectors $\tilde x\in\X$ that satisfy the
Galerkin type conditions:
\begin{equation}\label{sjdgsvd}
  \left\{\begin{aligned}
  A\tilde x-\tilde\alpha\tilde u&\perp\UU,\\
  B\tilde x-\tilde\beta\tilde v&\perp\V,\\
  \tilde\beta A^T\tilde u-\tilde\alpha B^T\tilde v&\perp\X.
  \end{aligned}\right.
\end{equation}
Among $k$ pairs $(\tilde\alpha,\tilde\beta)$'s, select
$\tilde\theta=\tilde\alpha/\tilde\beta$ closest to
$\tau$, and take $(\tilde\alpha,\tilde\beta,\tilde u,\tilde v,\tilde x)$
as an approximation to $(\alpha_*,\beta_*,u_*,v_*,x_*)$. We call
$(\tilde\alpha,\tilde\beta)$ or $\tilde\theta=\frac{\tilde\alpha}{\tilde\beta}$
a Ritz value and $\tilde u$, $\tilde v$
and $\tilde x$ the corresponding left and right Ritz vectors, respectively.

It follows from the thin QR factorizations \eqref{qrAXBX} of $A\widetilde X$ and
$B\widetilde X$
that $R_A=\widetilde U^TA\widetilde X$ and $R_B=\widetilde V^TA\widetilde X$.
Write $\tilde u=\widetilde U\tilde e$, $\tilde v=\widetilde V\tilde f$ and $\tilde x=\widetilde X\tilde d$.
Then (\ref{sjdgsvd}) becomes
\begin{equation}\label{sab}
	R_A\tilde d=\tilde\alpha\tilde e,\qquad
	R_B\tilde d=\tilde\beta\tilde f,\qquad
	\tilde\beta R_A^T\tilde e=\tilde\alpha R_B^T\tilde f,
\end{equation}
which is precisely the GSVD of the projected matrix pair $(R_A,R_B)$.
Therefore, in the extraction phase, the standard extraction
approach computes the GSVD of the $k$-by-$k$ matrix pair $(R_A,R_B)$,
picks up the GSVD component
$(\tilde\alpha,\tilde\beta,\tilde e,\tilde f,\tilde d)$
with $\tilde\theta=\frac{\tilde\alpha}{\tilde\beta}$
being the generalized singular value of $(R_A,R_B)$ closest
to the target $\tau$, and use
\begin{equation*}
  (\tilde\alpha,\tilde\beta,\tilde u,\tilde v,\tilde x)=
  (\tilde\alpha,\tilde\beta,\widetilde U\tilde e,\widetilde V\tilde f,\widetilde X\tilde d)
\end{equation*}
as an approximation to $(\alpha_*,\beta_*,u_*,v_*,x_*)$ of $(A,B)$.
It is straightforward from (\ref{sjdgsvd}) that $A\tilde x=\tilde\alpha\tilde u$,
$B\tilde x=\tilde\beta\tilde v$ and
$$
(A^TA-\tilde\theta^2\ B^TB)\tilde x\perp\X.
$$
That is, $(\tilde\theta^2,\tilde x)$ is a standard Ritz pair
to the eigenpair $(\sigma_*^2,x_*)$ of the symmetric definite
matrix pair $(A^TA,B^TB)$ with respect to the subspace $\X$.
Because of this, we call $(\tilde\alpha,\tilde\beta,\tilde u,\tilde v,\tilde x)$
a standard Ritz approximation in the GSVD context.

Since $A\tilde{x}=\tilde{\alpha} \tilde{u}$ and $B\tilde{x}=
\tilde{\beta}\tilde{v}$,
the residual of Ritz approximation $(\tilde{\alpha},\tilde{\beta},\tilde{u},
\tilde{v},\tilde{x})$ is
\begin{equation}\label{residual}
  r=r(\tilde{\alpha},\tilde{\beta},\tilde{u},\tilde{v},\tilde{x})=
  \tilde{\beta} A^T\tilde{u}-\alpha B^T\tilde{v}.
\end{equation}
Obviously, $(\tilde{\alpha},\tilde{\beta},\tilde{u},\tilde{v},\tilde{x})$
is an exact GSVD component of $(A,B)$ if and only if $\|r\|=0$.
The approximate GSVD component
$(\tilde{\alpha},\tilde{\beta},\tilde{u},\tilde{v},\tilde{x})$
is claimed to have converged if
\begin{equation}\label{converge}
\|r\|\leq (\tilde{\beta}\|A\|_1+\tilde{\alpha}\|B\|_1)\cdot tol,
\end{equation}
where $tol>0$ is a user prescribed tolerance,
and one then stops the iterations.

If $(\tilde{\alpha},\tilde{\beta},\tilde{u},\tilde{v},\tilde{x})$
has not yet converged, the CFP-JDGSVD method
expands the right searching subspace $\X$ and constructs the corresponding
left subspaces $\UU$ and $\V$ by (\ref{search}).
Specifically, the CPF-JDGSVD seeks for an expansion vector $t$ in the following
way: For the vector
\vspace{-1em}
\begin{equation}\label{defy}
  \tilde{y}:=(A^TA+B^TB)\tilde{x}=\tilde{\alpha} A^T\tilde{u}+
  \tilde{\beta} B^T\tilde{v}
\end{equation}
that satisfies $\tilde{y}^T\tilde{x}=1$, we first solve the correction equation
\begin{equation}\label{cortau}
  (I-\tilde{y}\tilde{x}^T)(A^TA-\rho^2B^TB)(I-\tilde{x}\tilde{y}^T)
  t=-r
\end{equation}
with the fixed $\rho=\tau$ for $t\perp \tilde{y}$ until
\begin{equation}\label{fixtol}
\|r\|\leq(\tilde{\beta}\|A\|_1+\tilde{\alpha}\|B\|_1)\cdot fixtol
\end{equation}
for a user prescribed tolerance $fixtol>0$, say, $fixtol=10^{-4}$,
and then solve the modified correction equation with
the dynamic $\rho=\tilde{\alpha}/\tilde{\beta}$ for $t\perp \tilde{y}$.
Note that $I-\tilde{y}\tilde{x}^T$ is an oblique projector onto
the orthogonal complement of the subspace ${\rm span}\{x\}$.

With the solution $t$ of \eqref{cortau}, we expand $\X$ to
the new $(k+1)$-dimensional $\X_{\rm new}={\rm span}\{\widetilde X,t\}$,
whose orthonormal basis matrix is
\begin{equation}\label{expandX}
  \widetilde X_{\mathrm{new}}=[\widetilde X,\  x_{+}]
\qquad\mbox{with}\qquad
 x_{+}=\frac{(I-\widetilde X\widetilde X^T)  t}{\|(I-\widetilde X\widetilde X^T) t\|},
\end{equation}
where $x_{+}$ is called an expansion vector. We then
compute the orthonormal bases $\widetilde U_{\new}$ and
$\widetilde V_{\new}$ of the expanded left searching subspaces
$$
\UU_{\new}=A\X_{\new}=\mathrm{span}\{\widetilde U,Ax_{+}\},\qquad
\V_{\new}=B\X_{\new}=\mathrm{span}\{\widetilde V,Bx_{+}\}
$$
by efficiently updating the thin QR factorizations of
$A\widetilde X_{\new}=\widetilde U_{\new}R_{A,\new}$ and
$B\widetilde X_{\new}=\widetilde V_{\new}R_{B,\new}$,
respectively, where
\begin{equation*}
  \widetilde U_{\mathrm{new}}=[\widetilde U, u_{+}], \quad
  R_{A,\mathrm{new}}=\begin{bmatrix}R_A & r_A\\&\gamma_A\end{bmatrix},\quad
  \widetilde V_{\mathrm{new}}=[\widetilde V, v_{+}], \quad
  R_{B,\mathrm{new}}=\begin{bmatrix}R_B & r_B\\&\gamma_B\end{bmatrix}
\end{equation*}
with
\begin{eqnarray*}
  r_A=\widetilde U^TAx_{+},\qquad
  \gamma_A=\|Ax_{+}-\widetilde Ur_A\|,\qquad
  u_{+}=\frac{Ax_{+}-\widetilde Ur_A}{\gamma_A},\\
  r_B=\widetilde V^TBx_{+},\qquad
  \gamma_B=\|Bx_{+}-\widetilde Vr_B\|,\qquad
  v_{+}=\frac{Bx_{+}-\widetilde Vr_B}{\gamma_A}.
\end{eqnarray*}

CPF-JDGSVD then computes a new approximate GSVD component of $(A,B)$
with respect to  $\UU_{\rm new}, \V_{\rm new}$ and $\X_{\rm new}$,
and repeat the above process until the convergence criterion
\eqref{converge} is achieved. We call iterative solutions of \eqref{cortau}
the inner iterations and the extractions of the approximate GSVD
components with respect to $\UU$, $\V$ and $\X$ the
outer iterations.

As has been shown in \cite{huang2020cross},
it suffices to iteratively solve the correction equations
approximately with low or modest accuracy and uses an approximate solution
to update $\X$ in the above way, in order that
the resulting {\em inexact} CPF-JDGSVD method and its {\em exact} counterpart with
the correction equations solved accurately behave similarly. Precisely,
for the correction equation \eqref{cortau}, we adopt the inner stopping criteria
in \cite{huang2020cross} and stop the inner
iterations when the inner relative residual norm $\|r_{in}\|$
satisfies
\hspace{-2em}
\begin{equation}\label{inncov}
  \|r_{in}\|\leq\min\{2c\tilde\varepsilon,0.01\},
\end{equation}
where $\tilde\varepsilon\in[10^{-4},10^{-3}]$ is a user prescribed
parameter and $c$ is a constant depending on $\rho$
and the current approximate generalized singular values.

\section{The harmonic extraction based JDGSVD methods}\label{sec:3}
We shall make use of the principle of
the harmonic extraction \cite{stewart2001matrix,vandervorst}
to propose the CPF-harmonic and IF-harmonic extraction based
JDGSVD methods in Section~\ref{subsec:3-1} and Section~\ref{subsec:3-2},
respectively. They compute new approximate GSVD components of $(A,B)$ with
respect to the given left and right searching subspaces $\UU$, $\V$ and $\X$,
and suit better for the computation of interior GSVD components.

\subsection{The CPF-harmonic extraction approach}\label{subsec:3-1}
If $B$ has full column rank with some special, e.g., banded, structure,
from which the inversion $(B^TB)^{-1}$ can be efficiently applied,
we can propose our CPF-harmonic extraction approach to compute a desired approximate
GSVD component as follows. For the purpose of derivation, assume that
\begin{equation}\label{choofBTB}
  B^TB=LL^T
\end{equation}
is the Cholesky factorization of $B^TB$ with $L\in\mathbb{R}^{n\times n}$ being
nonsingular and lower triangular, and define the matrix
\begin{equation}\label{deftildeA}
  \check{A}=AL^{-T}.
\end{equation}

We present the following result,
which establishes the relationship between the GSVD of $(A,B)$ and
the SVD of $\check A$ and will be used to propose the CPF-harmonic
extraction approach.

\begin{theoremmy}\label{theorem1}
Let $(\alpha_*,\beta_*,u_*,v_*,x_*)$ be a GSVD component of
the regular matrix pair $(A,B)$ and $\sigma_*=\frac{\alpha_*}{\beta_*}$.
Assume that $B$ has full column rank and $B^TB$ has
the Cholesky factorization \eqref{choofBTB}, and let
$\check A$ be defined by \eqref{deftildeA} and the vector
\begin{equation}\label{defz}
  z_*=\frac{1}{\beta_*}L^Tx_*.
\end{equation}
Then $(\sigma_*,u_*,z_*)$ is a singular triplet of $\check A$:
  \begin{equation}\label{svdtildea}
    \check Az_*=\sigma_* u_*
    \qquad\mbox{and}\qquad
    \check A^Tu_*=\sigma_* z_*.
  \end{equation}
\end{theoremmy}

\begin{proof}
It follows from the GSVD \eqref{gsvd} of $(A,B)$ that $Bx_*=\beta_* v_*$
with $\|v_*\|=1$, meaning that $\|Bx_*\|=\beta_*$.
Making use of \eqref{choofBTB}, we have
\begin{equation*}
  \|z_*\|=\frac{1}{\beta_*}\|L^Tx_*\|=\frac{1}{\beta_*}\|Bx_*\|=1.
\end{equation*}

By the definitions \eqref{deftildeA} and \eqref{defz} of
$\check A$ and $z_*$, from $Ax_*=\alpha_* u_*$
we obtain
\begin{equation*}
  \check Az_*=\frac{1}{\beta_*}AL^{-T}L^Tx_*
  =\frac{1}{\beta_*}Ax_*=\frac{\alpha_*}{\beta_*}u_*
  =\sigma_* u_*,
\end{equation*}
that is, the first relation in \eqref{svdtildea} holds.
From the GSVD \eqref{gsvd}, it is straightforward that
$A^Tu_*=\sigma_* B^Tv_*=\frac{\sigma_*}{\beta_*}B^TBx_*$.
Making use of this relation and \eqref{choofBTB} gives
\begin{equation*}
  \check A^Tu_*=L^{-1}A^Tu_*
  =\frac{\sigma_*}{\beta_*}L^{-1}B^TBx_*
  =\frac{\sigma_*}{\beta_*}L^Tx_*=\sigma_* z_*,
\end{equation*}
which proves the second relation in \eqref{svdtildea}.
\end{proof}

Theorem \ref{theorem1} motivates us to
propose our first harmonic extraction approach to
compute the singular triplet $(\sigma_*,u_*,z_*)$
of $\check A$ and then recover the desired GSVD component
$(\alpha_*,\beta_*,u_*,v_*,x_*)$ of $(A,B)$.

Specifically, take the $k$-dimensional $\UU$ and $\Z=L^T\X$ as the left and
right searching subspaces for the left
and right singular vectors $u_*$ and $z_*$ of $\check A$, respectively.
Then the columns of $\widetilde Z=L^T\widetilde X$ form a basis of $\Z$.
Mathematically, we seek for positive $\phi>0$ and
vectors $\check u\in\UU$ and $\check z\in\Z$
such that
\begin{equation}\label{cpfharmonic}
  \begin{bmatrix}
    0 &\check A^T\\\check A& 0
  \end{bmatrix}
  \begin{bmatrix}
    \check z\\\check u
  \end{bmatrix}
  -\phi
  \begin{bmatrix}
    \check z\\\check u
  \end{bmatrix}
  \ \perp\
  \left(\begin{bmatrix}
    0 &\check A^T\\\check A& 0
  \end{bmatrix}
  -\tau I\right)
\mathcal{R}\left(\begin{bmatrix}
   \widetilde Z&\\&  \widetilde U
  \end{bmatrix}\right).
\end{equation}
This is the harmonic extraction approach for the eigenvalue
problem of the augmented matrix
$$
\begin{bmatrix}
    0 &\check A^T\\\check A& 0
  \end{bmatrix}
$$
for the given target $\tau>0$ \cite{stewart2001matrix,vandervorst},
where $\phi$ is a harmonic Ritz value and $[\check z^T,\check u^T]^T$
is the harmonic Ritz vector with respect to the searching subspace
$$
\mathcal{R}\left(\begin{bmatrix}
   \widetilde Z&\\& \widetilde U
  \end{bmatrix}\right).
$$
We pick up the $\phi$ closest to $\tau$ as the approximation
to $\sigma_*$ and take the normalized $\check z/\|\check z\|$ and
$\check u/\|\check u\|$
as approximations to $z_*$ and $u_*$, respectively. We will show how
to obtain an approximation to $x_*$ afterwards.

Write $\check z=\widetilde Z\check d$ and $\check u=\widetilde U\check e$ with
$\check d\in\mathbb{R}^{k}$ and
$\check e\in\mathbb{R}^{k}$.
Then
$
\bsmallmatrix{\check z\\ \check u}
=\bsmallmatrix{\widetilde Z&\\&\widetilde U}
\bsmallmatrix{\check d\\ \check e}$,
and requirement \eqref{cpfharmonic} amounts to the equation
\begin{equation*}
  \begin{bmatrix}
    \widetilde Z^T&\\&\widetilde U^T
  \end{bmatrix}
  \begin{bmatrix}
    -\tau I &\check A^T\\\check A &-\tau I
  \end{bmatrix}
  \begin{bmatrix}
    -\phi I &\check A^T\\\check A &-\phi I
  \end{bmatrix}
  \begin{bmatrix}
    \widetilde Z&\\&\widetilde U
  \end{bmatrix}
  \begin{bmatrix}
    \check d\\\check e
  \end{bmatrix}=0.
\end{equation*}
Decompose $\phi=\tau+(\phi-\tau)$, and rearrange the above equation.
Then we obtain the generalized eigenvalue problem
of a $2k$-by-$2k$ matrix pair:
\begin{equation}\label{hproj}
  \begin{bmatrix}
    \widetilde Z^T\check A^T\check A\widetilde Z+\tau^2\widetilde Z^T\widetilde Z
    &-2\tau \widetilde Z^T\check A^T\widetilde U
    \\-2\tau \widetilde U^T\check A\widetilde Z
    &\widetilde U^T\check A\check A^T\widetilde U+\tau^2 I
  \end{bmatrix}\!
  \begin{bmatrix}\check d\\\check e\end{bmatrix}
  =(\phi-\tau)
  \begin{bmatrix}
    -\tau \widetilde Z^T\widetilde Z &  \widetilde Z^T\check A^T\widetilde U\\
    \widetilde U^T\check A\widetilde Z& -\tau I
  \end{bmatrix}\!
  \begin{bmatrix}\check d\\\check e\end{bmatrix}.
\end{equation}
By \eqref{deftildeA}, $\widetilde Z=L^T\widetilde{X}$ and the thin QR
factorization of $A\widetilde X$ in \eqref{qrAXBX}, we have
$$
\check A\widetilde Z=A\widetilde X=\widetilde UR_A,
$$
showing that
$$
\widetilde Z^T\check A^T\check A\widetilde Z=R_A^TR_A
\qquad\mbox{and}\qquad
\widetilde Z^T\check A^T\widetilde U=R_A^T.
$$
Moreover, exploiting the Cholesky factorization \eqref{choofBTB}
of $B^TB$ and the thin QR factorization of $B\widetilde X$ in \eqref{qrAXBX},
we obtain
\begin{eqnarray*}
\widetilde Z^T\widetilde Z&=&\widetilde X^TLL^T\widetilde X=
\widetilde X^TB^TB\widetilde X=R_B^TR_B,\\
\widetilde U^T\check A\check A^T\widetilde U&=&
\widetilde U^TA(LL^T)^{-1}A^T\widetilde U=
\widetilde U^TA(B^TB)^{-1}A^T\widetilde U.
\end{eqnarray*}
Substituting these two relations into \eqref{hproj} yields
\begin{multline}\label{cpfeq2}
  \begin{bmatrix}
    R_A^TR_A+\tau^2R_B^TR_B
    & -2\tau R_A^T
    \\-2\tau R_A
    &\widetilde U^TA(B^TB)^{-1}A^T\widetilde U+\tau^2I
  \end{bmatrix}
  \begin{bmatrix}
    \check d\\ \check e
  \end{bmatrix}\\
  =(\phi-\tau)
  \begin{bmatrix}
    -\tau R_B^TR_B &R_A^T\\R_A &-\tau I
  \end{bmatrix}
  \begin{bmatrix}
    \check d\\ \check e
  \end{bmatrix}.
\end{multline}

For the brevity of presentation, we will denote the symmetric
matrices
$$
H_{A,B^{\dag}}=\widetilde U^TA(B^TB)^{-1}A^T\widetilde U
$$
and
\begin{equation}\label{defGcHc}
  G_{\mathrm{c}}=\begin{bmatrix}
    -\tau R_B^TR_B
    &R_A^T \\R_A
    &-\tau I\!
  \end{bmatrix},
   \qquad
  H_{\mathrm{c}}=\begin{bmatrix}
    R_A^TR_A +\tau^2R_B^TR_B
    & -2\tau R_A^T \\
  -2\tau R_A
    &H_{A,B^{\dag}}+\tau^2I
  \end{bmatrix}.
\end{equation}
In implementations, we compute the generalized eigendecomposition
of the symmetric positive definite matrix pair $(G_{\mathrm{c}},H_{\mathrm{c}})$
and pick up the largest eigenvalue $\mu$ in magnitude and the
corresponding unit-length eigenvector $\bsmallmatrix{\check d\\\check e}$.
Then the harmonic Ritz approximation
to the desired singular triplet $(\sigma_*,\check u_*,\check z_*)$
of $\check A$ is
\begin{equation}\label{cpfh1st}
  (\phi,\check u,\check z)=\left(\tau+\frac{1}{\mu}, \frac{\widetilde U\check e}
  {\|\check e\|}, \frac{\widetilde Z\check d}{\|\widetilde Z\check d\|}\right).
\end{equation}

Since $\check z=\frac{\widetilde Z\check d}{\|\widetilde Z\check d\|}=\frac{L^T\widetilde X\check d}
{\|L^T\widetilde X\check d}\|$ is an approximation to the right singular vector
$z_*$ of $\check A$, from \eqref{defz}
the vector $L^{-T}\check z=\widetilde X\check d$ after
some proper normalization is an approximation to the right generalized singular
vector $x_*$ of $(A,B)$, which we write as
\begin{equation}\label{cpfx}
  \check x=\frac{1}{\check\delta}\widetilde X\check d,
\end{equation}
where $\check\delta$ is a normalizing factor.
It is natural to require
that the approximate right singular vector $\check x$ be
 $(A^T\!A+B^T\!B)$-norm normalized,
i.e., $\check x^T(A^T\!A+B^T\!B)\check x=1$, since the exact $x_*$ satisfies
$x_*^T(A^TA+B^TB)x_*=1$ by \eqref{gsvd}.
With this normalization, from \eqref{cpfx}, we have
\begin{equation*}
  1=\frac{1}{\check\delta^2}\check d^T\widetilde X^T(A^TA+B^TB)\widetilde X\check d
   =\frac{1}{\check\delta^2}\check d^T(R_A^TR_A+R_B^TR_B)\check d,
\end{equation*}
from which it follows that
\begin{equation}\label{checkdelta}
  \check\delta = \sqrt{\|R_A\check d\|^2+\|R_B\check d\|^2}.
\end{equation}

Note that the approximate left generalized singular vector
$\check u$ defined by \eqref{cpfh1st} is no longer collinear
with $A\check x$, as opposed to the collinear $\tilde u$ and
$A\tilde x$ obtained by the standard extraction approach in
Section~\ref{sec:2}. To this end, instead of $\check u$ in
\eqref{cpfh1st}, we take new $\check u$ and $\check v$ defined by
\begin{equation}\label{hritz}
\check u=\frac{A\check x}{\|A\check x\|}
\qquad\mbox{and}\qquad
\check v=\frac{B\check x}{\|B\check x\|}
\end{equation}
as the harmonic Ritz approximations to $u_*$ and $v_*$,
which are colinear with $A\check x$ and $B\check x$,
respectively.
Correspondingly, define $\check e=R_A\check d$ and $\check f=R_B\check d$.
Then by \eqref{checkdelta}, the parameter $\check\delta$
in \eqref{cpfx} becomes
$\check\delta=\sqrt{\|\check e\|^2+\|\check f\|^2}$.
Moreover, by definition \eqref{cpfx} of $\check x$
and the thin QR factorizations of $A\widetilde X$ and $B\widetilde X$ in
\eqref{qrAXBX}, we obtain
\begin{eqnarray*}
A\check x&=&\frac{1}{\check\delta}A\widetilde X\check d
=\frac{1}{\check\delta}\widetilde UR_A\check d
=\frac{1}{\check\delta}\widetilde U\check e,\\
B\check x&=&\frac{1}{\check\delta}B\widetilde X\check d
=\frac{1}{\check\delta}\widetilde VR_B\check d
=\frac{1}{\check\delta}\widetilde V\check f.
\end{eqnarray*}
Using them, we can efficiently compute the approximate generalized
singular vectors
\begin{equation}\label{cpfuv}
   \check u=\frac{A\check x}{\|A\check x\|}=\frac{\widetilde U\check e}{\|\check e\|}
    \qquad\mbox{and}\qquad
  \check v=\frac{B\check x}{\|B\check x\|}=\frac{\widetilde V\check f}{\|\check f\|}
\end{equation}
without forming products of the vector $\check x$ with the large $A$ and $B$.

As for the approximate generalized singular value $\phi$ in
\eqref{cpfh1st}, we replace it by the Rayleigh quotient
$\check\theta=\frac{\check \alpha}{\check\beta}$ of $(A,B)$
with respect to the approximate left and right generalized
singular vectors $\check u$ and  $\check v$, $\check x$, where
\begin{equation}\label{cpfab}
\check\alpha=\check u^TA\check x=\frac{\|\check e\|}{\check\delta}
\qquad\mbox{and}\qquad
\check\beta=\check v^TB\check x=\frac{\|\check f\|}{\check\delta} .
\end{equation}
The reason is that $\check\theta$ is a better approximation to $\sigma_*$
than the harmonic Ritz value $\phi$ in the sense that
\begin{equation}\label{rayqu}
  \|(A^TA-\check\theta^2B^TB)\check x\|_{(B^TB)^{-1}}
  \leq\|(A^TA-\phi^2B^TB)\check x\|_{(B^TB)^{-1}}.
\end{equation}

We remark that the residual of
$(\check\alpha,\check\beta,\check u,\check v,\check x)$  can
be defined similarly to (\ref{residual}), and a stopping
criterion similar to (\ref{converge}) can be used.

The CPF-harmonic extraction approach does not need to form
the cross product matrices $A^TA$ or $B^TB$ explicitly.
To distinguish from the approximation obtained by
the IF-harmonic extraction approach to
be proposed in the next subsection,
we call $(\check\alpha,\check\beta,\check u,\check v,\check x)$
the CPF-harmonic Ritz approximation to $(\alpha_*,\beta_*,u_*,v_*,x_*)$
with respect to the left and right searching subspaces
$\UU$, $\V$ and $\X$, where $(\check\alpha,\check\beta)$ or
$\check \theta$ is the CPF-harmonic Ritz value, and
$\check u,\check v$ and $\check x$
are the left and right CPF-harmonic Ritz vectors, respectively.
Particularly, if we expand $\UU,\V$ and $\X$ in a similar manner
to that described in Section~\ref{sec:2}, the resulting
method is called the CPF-harmonic JDGSVD method, abbreviated
as the CPF-HJDGSVD method.

From \eqref{cpfeq2}, we can efficiently update
the projected matrix pair $(G_{\mathrm{c}},H_{\mathrm{c}})$ as the subspaces are expanded.
At each expansion step, one needs to solve the
large symmetric positive definite linear
equations with the coefficient matrix $B^TB$ and the multiple
right-hand sides $A^T\widetilde U$. This can be done efficiently in parallel
whenever the Cholesky factorization (\ref{choofBTB}) of $B^TB$ can
be computed efficiently, which is the case for some structured
$B$, e.g., banded structure.

However, for a general large and sparse $B$, the calculation of
the Cholesky factorization (\ref{choofBTB}) of $B^TB$ may be costly
and even computationally infeasible. In this case, we can
compute $(B^TB)^{-1}A^T\widetilde U$ using
the Conjugate Gradient (CG) method for each column of $A^T\widetilde U$.
For $B$ well conditioned, the CG method converges fast.

Finally, we remark that, when $A$ is of full column rank,
the CPF-harmonic extraction approach proposed
above can be directly applied to the matrix pair $(B,A)$, whose GSVD
components are $(\beta_i,\alpha_i,v_i,u_i,x_i)$, $i=1,\dots,q$.

\subsection{The IF-harmonic extraction approach}\label{subsec:3-2}

As is clear from the previous subsection, CPF-HJDGSVD requires
that the symmetric $B^TB$ be positive definite, namely,
$B$ is square or rectangular and has full column rank.
If the direct application of $(B^TB)^{-1}$ is unaffordable or the CG
method converges slowly, then the CPF-harmonic extraction approach
is costly. Alternatively, we will propose an
inverse-free (IF) harmonic extraction approach that avoids this difficulty
and removes the above restriction on $B$.

Given the right searching subspace $\X$, the IF-harmonic extraction
approach seeks for an approximate generalized singular value
$\varphi>0$ and an approximate right generalized singular vector
$\hat x\in\X$ with $\|\hat x\|_{A^TA+B^TB}=1$ such that
\begin{equation}\label{harmonic}
  (A^TA-\varphi^2B^TB)\hat x\ \perp\ (A^TA-\tau^2B^TB)\X,
\end{equation}
namely, the residual of $(\varphi^2,\hat x)$ as an approximate
generalized eigenpair of the matrix pair $(A^TA,B^TB)$
is orthogonal to the subspace $(A^TA-\tau^2B^TB)\X$. This is precisely the harmonic
Rayleigh--Ritz projection of  $(A^TA,B^TB)$ onto $\X$ with respect to the target
$\tau^2$, and the $k$ pairs $(\varphi^2,\hat x)$ are the
harmonic Ritz approximations of
$(A^TA,B^TB)$ with respect to $\X$ for the given $\tau^2$.
One selects the positive $\varphi$ closest to $\tau$ and
the corresponding $\hat x$ as approximations to the desired
generalized singular value $\sigma$ closest to $\tau$ and the corresponding
right generalized singular vector $x$.

Since the columns of $(A^TA-\tau^2B^TB)\widetilde X$ span the
subspace $(A^TA-\tau^2B^TB)\X$,
requirement \eqref{harmonic} is equivalent to
\begin{eqnarray*}
\widetilde X^T(A^TA-\tau^2B^TB)(A^TA-\varphi^2B^TB)\widetilde X\hat d=0
\qquad\mbox{with}\qquad
\hat x=\frac{1}{\hat\delta}\widetilde X\hat d,
\end{eqnarray*}
where $\hat\delta$ is a normalizing factor such that
$\|\hat x\|_{A^TA+B^TB}=1$.

Writing $\varphi^2=\tau^2+(\varphi^2-\tau^2)$ and
rearranging the above equation, we obtain
\begin{equation}\label{harmoniceq}
  \widetilde X^T(A^TA-\tau^2B^TB)^2\widetilde X\hat d
  =(\varphi^2-\tau^2)\widetilde X^T(A^TA-\tau^2B^TB)
  B^TB\widetilde X\hat d,
\end{equation}
that is, $\mu=\varphi^2-\tau^2$ is a generalized eigenvalue of the
matrix pair $(H_{\tau},G_{\tau})$ and $\hat d$ is the corresponding
normalized generalized eigenvector, where
 \begin{equation}\label{defGiHi}
  G_{\tau}=\widetilde X^T(A^TA-\tau^2B^TB)B^TB\widetilde X
  \qquad\mbox{and}\qquad
  H_{\tau}=\widetilde X^T(A^TA-\tau^2B^TB)^2\widetilde X.
\end{equation}
We compute the generalized eigendecomposition of $(G_{\tau},H_{\tau})$,
pick up its largest generalized eigenvalue $\nu=\frac{1}{\mu}$ in magnitude,
and take $\varphi=\sqrt{\tau^2+\frac{1}{\nu}}$ as an approximation to $\sigma_*$.
Correspondingly, the harmonic Ritz pair to approximate $(\sigma_*,x_*)$ is
\begin{equation}\label{IFx}
(\varphi,\hat x)=\left(\sqrt{\tau^2+\frac{1}{\nu}},
\frac{1}{\hat\delta}\widetilde X\hat d\right),
\end{equation}
where $\hat d$ is the generalized eigenvector of $(G_{\tau},H_{\tau})$
corresponding to the eigenvalue $\nu$.

As for the normalizing factor $\hat\delta$, by the requirement that
$\|\hat x\|_{A^TA+B^TB}$ $=1$, following the same derivations
as in Section~\ref{subsec:3-1}, we have
\begin{equation}\label{delta}
  \hat\delta=\sqrt{\|\hat e\|^2+\|\hat f\|^2}
  \quad\mbox{with}\quad
  \hat e=R_A \hat d
  \quad\mbox{and}\quad
  \hat f=R_B \hat d,
\end{equation}
where $R_A$ and $R_B$ are defined by \eqref{qrAXBX}.
Analogously to that done in Section~\ref{subsec:3-1},
rather than using the harmonic Ritz value $\varphi$ to
approximate $\sigma_*$,
we recompute a new and better approximate generalized singular
value and the corresponding left generalized singular vectors by
\begin{equation}\label{newapp}
  \hat\alpha=\|A\hat x\|,\qquad
    \hat\beta=\|B\hat x\|
  \qquad\mbox{and}\qquad
   \hat u=\frac{A\hat x}{\|A\hat x\|},\qquad
   \hat v=\frac{B\hat x}{\|B\hat x\|}.
\end{equation}
Since the new approximate generalized singular value
$\hat\theta=\frac{\hat\alpha}{\hat\beta}$ is the square root
of the Rayleigh quotient of the matrix pair $(A^TA,B^TB)$ with respect to $\hat x$,
as an approximation to $\sigma_*$, it is more accurate than $\varphi$ in
(\ref{IFx})  in the sense of (\ref{rayqu}) when the CPF-harmonic approximations $\check x$, $\check \theta$ and
$\phi$ are replaced by the IF-harmonic ones $\hat x$, $\hat\theta$ and $\varphi$, respectively.

It is straightforward to verify that
$(\hat\alpha,\hat\beta,\hat u,\hat v,\hat x)$ in
\eqref{newapp} satisfies
$A\hat x=\hat\alpha u$ and $B\hat x=\hat\beta\hat v$ with
$\|\hat u\|=\|\hat v\|=1$ and $\hat\alpha^2+\hat\beta^2=1$.
By \eqref{qrAXBX}, \eqref{delta} and \eqref{newapp}, it is easily shown that
\begin{equation}\label{IFuv}
   \hat\alpha=\frac{\|\hat e\|}{\hat\delta},\qquad
   \hat\beta=\frac{\|\hat f\|}{\hat\delta}
   \qquad\mbox{and}\qquad
   \hat u=\frac{\widetilde U\hat e}{\|\hat e\|},\qquad
   \hat v=\frac{\widetilde V\hat f}{\|\hat f\|}.
\end{equation}
Therefore, compared with (\ref{newapp}),
we can exploit formula (\ref{IFuv}) to
compute $\hat\alpha,\hat\beta$ and $\hat u,\hat v$ more efficiently without
using $A$ and $B$ to form matrix-vector products.
We call $(\hat\alpha,\hat\beta,\hat u,\hat v,\hat x)$
the IF-harmonic Ritz approximation to $(\alpha_*,\beta_*,u_*,v_*,x_*)$
with respect to the left and right searching subspaces
$\UU$, $\V$ and $\X$, where the pair $(\hat\alpha,\hat\beta)$ or
$\hat\theta=\frac{\hat\alpha}{\hat\beta}$ is the IF-harmonic Ritz value, and
$\hat u,\hat v$ and $\hat x$
are the left and right IF-harmonic Ritz vectors, respectively.
Particularly, when expanding $\UU,\V$ and $\X$ in a similar manner
to that described in Section~\ref{sec:2}, the resulting
method is called the IF-harmonic JDGSVD method, abbreviated as the IF-HJDGSVD
method.

Based on the way that $(\hat\alpha,\hat\beta,\hat u,\hat v,\hat x)$
is computed, the associated residual and stopping criterion
are defined as (\ref{residual}) and designed as (\ref{converge}), respectively.

In computations, as $\UU,\V$ and $\X$
are expanded, we first update the intermediate matrices
\begin{equation}\label{defHAB}
  H_A=\widetilde X^T(A^TA)^2\widetilde X,\quad
  H_B=\widetilde X^T(B^TB)^2\widetilde X,\quad
  H_{A,B}=\widetilde X^TA^TAB^TB\widetilde X
\end{equation}
efficiently and then form the matrices
\begin{equation}\label{comGiHi}
  G_{\tau}=H_{A,B}-\tau^2H_{B}
  \qquad\mbox{and}\qquad
  H_{\tau}=H_A+\tau^4H_B-\tau^2(H_{A,B}^T+H_{A,B}).
\end{equation}

Compared with the CPF-harmonic extraction, the IF-harmonic extraction
does not involve $(B^TB)^{-1}$.
Note that it uses $B^TB$ and $A^TA$ explicitly when
forming the matrices $G_{\tau}$ and $H_{\tau}$ in \eqref{defGiHi}.
Fortunately, provided that the desired $\sigma_*$ is not very small,
then $\sigma_*^2$ is a well conditioned eigenvalue of $(A^TA,B^TB)$
and $\hat\theta$ is an approximation to $\sigma_*$ with
the accuracy  $\|(A^TA-\hat\theta^2 B^TB)\hat x\|$
\cite[Sect. 3, Chap. XI]{stewart90}.

\section{Thick-restart JDGSVD type algorithms with deflation and
purgation} \label{sec:5}

As the subspace   dimension $k$ increases, the
computational complexity of the proposed JDGSVD type algorithms will
become prohibitive. For a maximum number $k=k_{\max}$ allowed, if
the algorithms do not yet converge, then it is necessary to restart
them. In this section, we show how to effectively and
efficiently restart CPF-HJDGSVD and IF-HJDGSVD proposed in
Section~\ref{sec:3}, and how to introduce some efficient
novel deflation and purgation techniques into them to compute
more than one, i.e., $\ell>1$, GSVD components of $(A,B)$.

\subsection{Thick-restart}\label{subsec:5-1}
We adopt a commonly used thick-restart
technique, which was initially advocated in \cite{stath1998}
and has been popularized in a number of papers, e.g.,
\cite{huang2019inner,huang2020cross,stath1998,wu16primme,wu2015preconditioned}.
Adapting it to our case, we take three new initial searching subspaces to
be certain $k_{\min}$-dimensional subspaces of the left and right searching
subspaces at the current cycle,
which aim to contain as much information as possible on the desired left and
right generalized singular vectors and their few neighbors.
Then we continue to expand the subspaces in the regular way described in
Sections~\ref{sec:2}--\ref{sec:3}, and compute new approximate
GSVD components with respect to the expanded subspaces.
We check the convergence at each step, and if converged, stop; otherwise
expand the subspaces until the subspace dimension reaches $k_{\max}$.
Proceed in this way until the desired GSVD component is
found. In what follows we describe how to
efficiently implement thick-restart in our GSVD context, which turns out to
be involved and is not as direct as in the context of
the standard eigenvalue problem and SVD problem.

At the current extraction phase, either CPF-HJDGSVD or IF-HJDGSVD
has computed $k_{\min}$ approximate right generalized singular vectors,
denoted by $\tilde{x}_i\!=\widetilde Xd_i$ in a unified form, corresponding to
the $k_{\min}$ approximate generalized singular
values closest to $\tau$, where $\tilde{x}_1$
is used to approximate the desired $x_*$.
Write $\widetilde{X}_1=[\tilde{x}_1,\dots,\tilde{x}_{k_{\min}}]$
and $D_1=[d_1,\dots,d_{k_{\min}}]$,
and take the new initial right searching subspace
$$
\X_{\rm new}={\rm span}\{\widetilde{X}_1\}
={\rm span}\{\widetilde XD_1\}.
$$
Compute the thin QR factorization of $D_1$ to obtain its Q-factor
$Q_d\in\mathbb{R}^{k_{\max}\times k_{\min}}$.
Then the columns of
\begin{equation}\label{xnew}
\widetilde X_{\new}=\widetilde XQ_d
\end{equation}
form an orthonormal basis of $\X_{\new}$. Correspondingly, we take
the new initial left subspaces
$\UU_{\new}=A\X_{\new}$ and $\V_{\new}=B\X_{\new}$.
Notice that
$$
A\widetilde X_{\new}=A\widetilde XQ_d=\widetilde UR_AQ_d
\qquad\mbox{and}\qquad
B\widetilde X_{\new}=B\widetilde XQ_d=\widetilde VR_BQ_d.
$$
We compute the thin QR factorizations of the small matrices $R_AQ_d$
and $R_BQ_d$:
\begin{equation*}
R_AQ_d=Q_eR_{A,\new}
\qquad\mbox{and}\qquad
R_BQ_d=Q_fR_{B,\new},
\end{equation*}
where $Q_e,Q_f\in\mathbb{R}^{k_{\max}\times k_{\min}}$
are orthonormal, and
$R_{A,\new}$ and $R_{B,\new}\in\mathbb{R}^{k_{\min}\times k_{\min}}$
are upper triangular.
Then the columns of
\begin{equation*}
\widetilde U_{\new}=\widetilde UQ_e
\qquad\mbox{and}\qquad
\widetilde V_{\new}=\widetilde VQ_f
\end{equation*}
form orthonormal bases of $\UU_{\new}$ and $\V_{\new}$,
and $R_{A,\new}$ and $R_{B,\new}$ are the R-factors of
$A\widetilde X_{\new}=\widetilde U_{\new}R_{A,\new}$ and $B\widetilde X_{\new}=\widetilde V_{\new}R_{B,\new}$, respectively.

For the CPF-harmonic extraction,
we need to update the projection matrices $G_{\mathrm{c}}$
and  $H_{\mathrm{c}}$ defined by \eqref{defGcHc}.
Concretely, we compute $G_{\mathrm{c},\new}$ and the $(1,1)$-, $(1,2)$- and
$(2,1)$-block submatrices of $H_{\mathrm{c},\new}$
by using $R_{A,\new}$ and $R_{B,\new}$.
The $(2,2)$-block submatrix
$H_{\mathrm{c2},\new}=H_{A,B^{\dag},\new}+\tau^2I$ of $H_{\mathrm{c},\new}$ is updated efficiently without involving $(B^TB)^{-1}$:
\begin{equation}\label{update}
H_{A,B^{\dag},\new}=\widetilde U_{\new}^TA(B^TB)^{-1}A^T\widetilde U_{\new}=
Q_e^T{H_{A,B^{\dag}}}Q_e
\end{equation}
where $H_{A,B^{\dag}}=\widetilde U^TA(B^TB)^{-1}A^T\widetilde U$ is part of the $(2,2)$-block submatrix of
$H_{\mathrm{c}}$. For the IF-harmonic extraction, we efficiently update the intermediate
matrices $H_A$, $H_B$ and $H_{A,B}$ in \eqref{defHAB} by
\begin{equation*}
  H_{A,\new}=Q_d^TH_AQ_d,\qquad
  H_{B,\new}=Q_d^TH_BQ_d,\qquad
  H_{A,B,\new}=Q_d^TH_{A,B}Q_d.\qquad
\end{equation*}

\subsection{Deflation and purgation}
If the GSVD components $(\alpha_i,\beta_i,u_i,v_i,x_i)$, $i\!=\!1,\dots,\ell$
of $(A,B)$ are required with $\sigma_i=\frac{\alpha_i}{\beta_i}$ labeled as in
(\ref{order}), we can adapt the efficient
deflation and purgation techniques in \cite{huang2020cross} to our JDGSVD algorithms.

Assume that the $j$ approximate GSVD components
$(\alpha_{i,c},\beta_{i,c},u_{i,c},v_{i,c},x_{i,c})$
have converged to the desired GSVD components
$(\alpha_{i},\beta_{i},u_{i},v_{i},x_{i})$ with
\begin{equation}\label{stopcrit}
\|r_i\|=\|\beta_{i,c} A^Tu_{i,c}-\alpha_{i,c}B^Tv_{i,c}\|
\leq(\beta_{i,c}\|A\|_1+\alpha_{i,c}\|B\|_1)\cdot tol, \qquad i=1,\dots,j.
\end{equation}
Write  $C_c=\diag\{\alpha_{1,c},\dots,\alpha_{j,c}\}$,
$S_c=\diag\{\beta_{1,c},\dots,\beta_{j,c}\}$ and
$U_c=[u_{1,c},\dots, u_{j,c}]$,
$V_c=[v_{1,c},\dots, v_{j,c}]$,
$X_c=[x_{1,c},\dots, x_{j,c}]$.
Then $(C_c,S_c,U_c,V_c,X_c)$ is a converged approximate partial
GSVD of $(A,B)$ that satisfies
$$
AX_c=U_cC_c,\qquad  BX_c=V_cS_c, \qquad C_c^2+S_c^2=I_j
$$
and
\begin{equation*}
  \|R_c\|_F=\|A^TU_cS_c-B^TV_cC_c\|_F\leq\sqrt{j(\|A\|_1^2+\|B\|_1^2)}\cdot tol.
\end{equation*}
Proposition~4.1 of \cite{huang2020cross} proves that
if $tol=0$ in (\ref{stopcrit}) then the exact nontrivial GSVD components of
the modified matrix pair
\begin{equation}\label{defpair}
(A(I-X_cY_c^T),B(I-X_cY_c^T))
\qquad\mbox{with}\qquad
Y_{c}=(A^TA+B^TB)X_c
\end{equation}
are $(\alpha_i,\beta_i,u_i,v_i,x_i),\ i=j+1,\ldots,q$, where
$Y_{c}$ satisfies $X_c^TY_c=I_{j}$.
Therefore, we can apply either CPF-HJDGSVD or IF-HJDGSVD to
the pair \eqref{defpair}, and compute the next desired GSVD component
$(\alpha_{*},\beta_{*},u_{*},v_{*},x_{*}):
=(\alpha_{j+1},\beta_{j+1},u_{j+1},v_{j+1},x_{j+1})$.

To this end, we require that the converged $X_c$ and $Y_c$ be bi-orthogonal,
i.e., $X_c^TY_c=I$. Moreover, as the right searching subspace $\X$ is expanded,
we require that $\X$ be always $(A^TA+B^TB)$-orthogonal
to the converged approximate
right generalized singular vectors $x_{1,c},\dots,x_{j,c}$, i.e.,
$\widetilde X^TY_c=\bm{0}$.
Such an orthogonality can be guaranteed
in computations, as shown below.

Assume that $X_c^TY_c=I_j$ and $\widetilde X^T\perp Y_c$.
At the extraction phase, we use the CPF-harmonic or IF-harmonic extraction
to obtain an approximate GSVD component $(\alpha,\beta,u,v,x)$ of $(A,B)$.
If $(\alpha,\beta,u,v,x)$ has not yet converged,
we construct $X_p=[X_c,x]$ and $Y_p=[Y_c,y]$ with
$y=(A^TA+B^TB)x=\alpha A^Tu+\beta B^Tv$. Then it follows
from $X_c^TY_c=I_{j}$ and $x\perp Y_c$ that
$X_c^Ty=Y_c^Tx=\bm{0}$, $x^Ty=1$ and
$X_p^TY_p=I_{j+1}$. Therefore, $I-Y_pX_p^T$ is an oblique projector.
At the subspace expansion phase, instead of \eqref{cortau},
we use an iterative solver, e.g., the MINRES method \cite{saad2003},
to approximately solve
the modified symmetric correction equation
\begin{equation}\label{deflat}
  (I-Y_pX_p^T)(A^TA-\rho^2B^TB)(I-X_pY_p^T)t=-(I-Y_pX_p^T)r
  \qquad\mbox{for}\qquad
  t\perp Y_p
\end{equation}
with $r$ being the residual (\ref{residual})
of $(\alpha,\beta,u,v,x)$ and $\rho=\tau$ or $\frac{\alpha}{\beta}$.
Having found an approximate solution $\tilde{t}\perp Y_p$, we orthonormalize
it against $\widetilde X$ to obtain the expansion vector $x_{+}$ and
update $\widetilde X$ by \eqref{expandX}.
By assumption and \eqref{deflat}, both $\widetilde X$ and $\tilde t$ are orthogonal to $Y_c$,
which makes the expansion vector $x_{+}$ and the
expanded right searching subspace $\X$ orthogonal to $Y_c$.

If $(\alpha,\beta,u,v,x)$ has already converged, we add it to
the converged partial GSVD $(C_c,S_c,U_c,V_c,X_c)$ and set $j:=j+1$.
By assumption, the old $X_c$ and $Y_c$ are bi-orthogonal.
	Since the added $x$ is  orthogonal to the old $Y_c$, it is known
	that the new $X_c$ and $Y_c$ are also bi-orthogonal.
Proceed in this way until all the $\ell$ desired GSVD components of $(A,B)$
are found.

Remarkably, when $(\alpha,\beta,u,v,x)$ has converged,
the current searching subspaces usually contain reasonably
good information on the next desired GSVD component.
In order to make full use of such available information when computing
the next $(\alpha_{*},\beta_{*},u_{*},v_{*},x_{*})$, the authors
in \cite{huang2019inner,huang2020cross}
have proposed an effective and efficient purgation strategy. It can
be adapted to our current context straightforwardly:
We purge the newly converged $x=\widetilde Xd$ from the current $\X$
and take the reduced subspace $\X_{\new}$ as the initial right
searching subspace for computing the next desired GSVD component
of $(A,B)$.
To achieve this, we compute the QR factorization of the $k\times 1$ matrix
$d^{\prime}=(R_A^TR_A+R_B^TR_B)d$ to obtain its Q-factor
$\left[\frac{d^{\prime}}{\|d^{\prime}\|},Q_D\right]$ such that the columns of
$Q_D\in\mathbb{R}^{k\times (k-1)}$ form an orthonormal basis of the
orthogonal completement subspace of ${\rm span}\{d^{\prime}\}$.
Then the columns of
$$\widetilde X_{\new}=\widetilde XQ_D$$
form an orthonormal basis of $\X_{\new}$, and $\widetilde X_{\new}$ is
orthogonal to $Y_{c,\new}=[Y_c,y]$ with $y=(A^TA+B^TB)x$
because $\widetilde X_{\new}^TY_c=\bm{0}$ and
\begin{equation*} \widetilde X_{\new}^Ty=Q_D^T\widetilde X^T(A^TA+B^TB)\widetilde Xd=Q_D^T(R_A^TR_A+R_B^TR_B)d=Q_D^Td^{\prime}=\bm{0}.
\end{equation*}

Therefore, provided that $Q_d$ in \eqref{xnew}
is replaced by $Q_D$, just as done in Section~\ref{subsec:5-1}, we can
efficiently construct orthonormal base of the new initial searching left
and right subspaces $\mathcal{U}_{\rm new}$, $\mathcal{V}_{\rm new}$ and
$\mathcal{X}_{\rm new}$.
Therefore, the purgation can be done with very little
cost. We then continue to expand the subspaces in a regular way
until their dimensions reach $k_{\max}$.

\section{Numerical experiments}\label{sec:6}
In this section, we report numerical experiments on several problems
to illustrate the performance of the two harmonic extraction based algorithms CPF-HJDGSVD, IF-HJDGSVD and the standard extraction based algorithm CPF-JDGSVD in \cite{huang2020cross}, and make a comparison of them.
All the numerical experiments were performed on an Intel Core (TM)
i9-10885H CPU 2.40 GHz with 64 GB RAM using the Matlab R2021a with the
machine precision $\epsilon_{\mathrm{mach}}=2.22\times10^{-16}$ under the
Miscrosoft Windows 10 64-bit system.

\begin{table}[tbhp]
{\small\caption{Basic properties of the test matrix pairs.}\label{table0}
\begin{center}
\begin{tabular}{ccccccccc} \toprule
 $A$&$B$&$m$&$p$&$n$&$nnz$&$\kappa(\bsmallmatrix{A\\B})$
 &$\sigma_{\max}$&$\sigma_{\min}$ \\ \midrule
 nd3k&$T$&9000&9000&9000&3306688&9.33e+1&1.16e+2&1.77e-6\\
 viscoplastic1&T&4326&4326&4326&74142&7.39e+1&5.26e+1&1.51e-4\\
 rajat03&$T$&7602&7602&7602&55457&5.10e+2&2.65e+2&8.07e-6\\
 $\mathrm{lp\_bnl2}^T$&$T$&4486&2324&2324&21966&1.93e+2&1.10e+2&1.20e-2\\
 Hamrle2&$T$&5952&5952&5952&40016&1.04e+2&7.29e+1&4.12e-4\\
 $\mathrm{jendrec1}^T$&$T$&4228&2109&2109&95933&8.95e+2&1.86e+3&7.86e-1\\

grid2&$L_1$&3296&3295&3296&19454&7.54e+1&1.93e+3&3.32e-17\\
 dw1024&$L_1$&2048&2047&2048&14208&8.03&5.25e+2&2.55e-4\\
 $\mathrm{r05}^T$&$L_1$&9690&5189&5190&114523&6.24e+1&1.19e+4&2.91e-1\\
 $\mathrm{p05}^T$&$L_1$&9590&5089&5090&69223&4.40e+1&9.77e+3&2.91e-1\\

bibd\_81\_2&$L_2$&3240&3238&3240&12954&4.12&4.69e+5&2.50e-1\\
 benzene&$L_2$&8219&8217&8219&267320&5.58e+2&1.60e+7&2.89e-1\\
 blckhole&$L_2$&2132&2130&2132&21262&3.64e+1&1.32e+6&6.90e-4\\
\bottomrule
\end{tabular}
\end{center}}
\end{table}

Table~\ref{table0} lists all the test problems together with some
of their basic properties,
where the matrices $A$ or their transpose(s) are sparse matrices
from \cite{davis2011university} with $m\geq n$,
the matrices $B$ are taken to be (\rmnum 1) the symmetric tridiagonal Toeplitz
matrices $T$ with $p=n$ whose diagonal and subdiagonal elements are $3$ and $1$,
respectively, and (\rmnum 2)
\begin{equation*}
  L_1=\begin{bmatrix}
      1&-1&&\\
      &\ddots&\ddots&\\&&1&-1
    \end{bmatrix}
    \qquad\mbox{and}\qquad
 L_2=\begin{bmatrix}
      -1&2&-1&&\\
      &\ddots&\ddots&\ddots&\\
      &&-1&2&-1
    \end{bmatrix},
\end{equation*}
which are the scaled discrete approximations of the
first and second order derivative operators in dimension one
with $p=n-1$ and $p=n-2$, respectively,
$nnz$ denotes the total numbers of the nonzero elements in $A$ and $B$,
and $\sigma_{\max}$ and $\sigma_{\min}$ denote the largest and smallest
nontrivial generalized singular values of $(A,B)$, respectively.
We mention that, for those matrix pairs $(A,B)$ with $B=T$, all the
generalized singular values of $(A,B)$ are nontrivial ones and,
for the matrix pairs $(A,B)$ with $B=L_1$ and $L_2$, there are one
and two infinite generalized singular values, respectively.

For the three algorithms under consideration, we take the
vectors ${\sf ones}(n,1)$ and ${\sf mod}(1:n,4)$ and normalize them
to form one dimensional right searching subspaces
for $(A,B)$ with $B=T$ and $B=L_i,\ i=1,2$, respectively,
where ${\sf ones}$ and ${\sf mod}$ are the Matlab built-in functions.
When the dimensions of $\X$, $\UU$ and $\V$ reach the maximum number
$k_{\max}=30$ but the algorithms do not converge,
we use the corresponding thick-restart algorithms by taking $k_{\min}=3$.
An approximate GSVD component
is claimed to have converged if its relative residual norm satisfies
\eqref{converge} with $tol=10^{-8}$. We stop the algorithms if
all the $\ell$ desired GSVD components have been computed successfully
or the total $K_{\max}=n$ outer iterations have been used.
For the correction equation \eqref{cortau},
we first take $\rho=\tau$ and then switch to $\rho=\theta$
if the outer residual norm satisfies \eqref{fixtol} with $fixtol=10^{-4}$.
We take zero vectors as initial solution guesses for the inner iterations
and use the Matlab built-in function {\sf minres}
to solve the correction equation \eqref{cortau} or \eqref{deflat}
until the inner relative residual norm meets \eqref{inncov} with the
stopping criterion $\tilde\varepsilon=10^{-4}$. We comment
that, as our extensive experience has demonstrated,
preconditioning the correction equations by ILU type
factorizations \cite{saad2003} has turned out to be ineffective and does not
reduce the inner iterations for most of the test problems.
The ineffectiveness is due to the (high) indefiniteness of correction equations.
Therefore, we report only the results using
the MINRES method without preconditioning.

In all the tables, for the ease of presentation, we
further abbreviate the CPF-JDGSVD, CPF-HJDGSVD and IF-HJDGSVD
algorithms as CPF, CPFH and IFH, respectively.
We denote by $I_{out}$ and $I_{in}$ the total numbers of
outer and inner iterations that an underlying JDGSVD algorithm
uses to achieve the convergence, respectively,
and by $T_{cpu}$ the total CPU time in seconds counted by
the Matlab built-in commands {\sf tic} and {\sf toc}.

\begin{exper}\label{exp1}
  We compute one GSVD component of $(A,B)=(\mathrm{nd3k},T)$
  associated with the generalized singular value closest to
  the target $\tau=10$ that is highly clustered with some other
  ones of $(A,B)$.
\end{exper}

\begin{figure}[tbhp]
	\centering
	\includegraphics[width=0.90\textwidth]{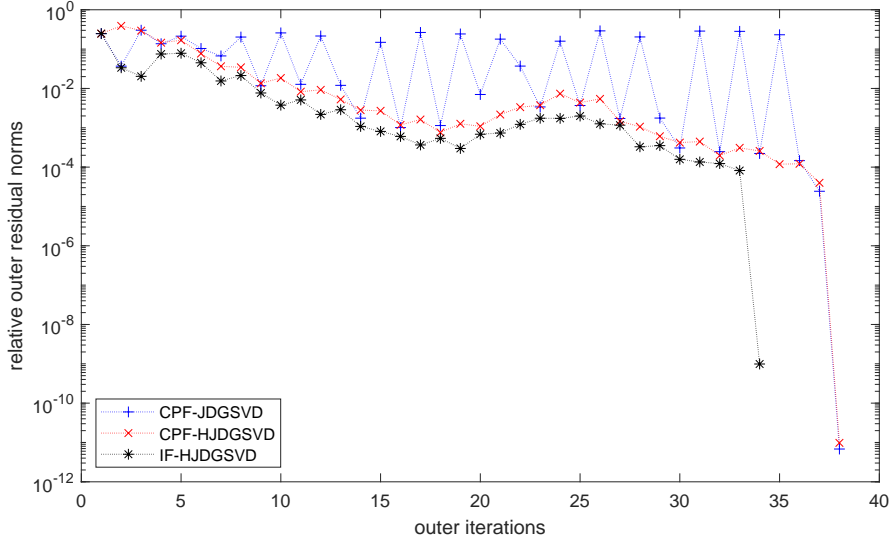}
	\caption{Computing one GSVD component of
		$(A,B)=(\mathrm{nd3k},T)$ with $\tau=10$.}\label{fig1}
\end{figure}

For the matrix pairs $(A,B)$ in this and the next two examples,
the matrices $B=T$'s are well conditioned,
and their Cholesky factorizations can be cheaply computed
at the cost of $\mathcal{O}(n)$ flops, so that
each matrix-vector product with $B^{-1}$ can be implemented
using $\mathcal{O}(n)$ flops.
Therefore, at the expansion phase of each step of CPF-HJDGSVD, we use the Matlab recommended command {$\backslash$} to carry out $B^{-1}$-vector products
and update the matrix $H_{A,B^{\dag}}$ in \eqref{defGcHc}.
Purely for the experimental purpose and the illustration of
the truly convergence behavior, when solving the inner
linear systems \eqref{cortau} involved in the JDGSVD type algorithms,
we compute the LU factorizations of $A^TA-\rho^2B^TB$ and use them to
solve the linear systems.
That is, we solve all the correction equations accurately in finite
precision arithmetic. We will demonstrate how the CPF-harmonic and
IF-harmonic extraction approaches behave.
Figure~\ref{fig1} depicts the outer convergence curves of the three
JDGSVD type algorithms.

As can be seen from Figure~\ref{fig1}, compared with CPF-JDGSVD,
the two harmonic JDGSVD algorithms have smoother
outer convergence behavior, and IF-HJDGSVD
uses four fewer outer iterations to reach the convergence
than CPF-JDGSVD and CPF-HJDGSVD. This illustrates the advantage
of IF-HJDGSVD over CPF-JDGSVD and CPF-HJDGSVD. Of
CPF-JDGSVD and CPF-HJDGSVD, although they use
the same number of outer iterations to converge, CPF-HJDGSVD should be
favorable because of its much more regular convergence behavior.

\begin{exper}\label{exp2}
  We compute one GSVD component  of $(A,B)=(\mathrm{viscoplastic1},T)$
  with the generalized singular value closest to a small target $\tau= 6.7e-2$
  being clustered with some other ones of $(A,B)$. We should
  notice that $\tau$ is fairly near to the left-end point $\sigma_{\min}=1.51e-4$
  of the generalized singular spectrum of $(A,B)$. This implies
  that the desired generalized singular vectors and
  the correction equations \eqref{cortau} are ill conditioned,
  causing that {\sf minres} converges slowly.
\end{exper}

\begin{figure}[tbhp]
	\centering
	\includegraphics[width=0.90\textwidth]{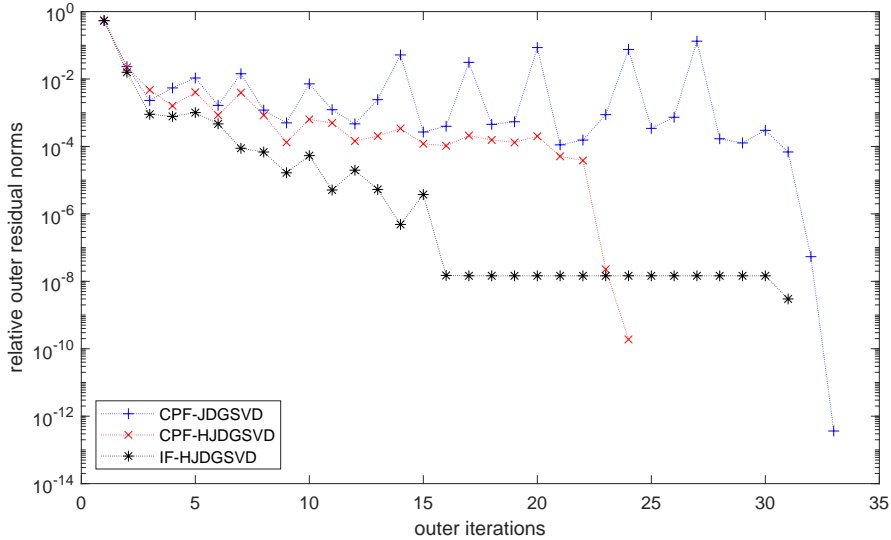}
	\caption{Computing one GSVD component of
		$(A,B)=(\mathrm{viscoplastic1},T)$ with $\tau=6.7e-2$.}\label{fig2}
\end{figure}

We draw the outer convergence curves of the three
JDGSVD type algorithms in Figure~\ref{fig2}.
As the figure shows, both CPF-HJDGSVD and IF-HJDGSVD converge much more
regularly than CPF-JDGSVD. Specifically, IF-HJDGSVD converges much faster
than CPF-JDGSVD in the first sixteen outer iterations, and
it has already reached the level of $\mathcal{O}(10^{-8})$ at iteration 16.
Although it stagnates for the next several outer iterations,
IF-HJDGSVD manages to converge two outer iterations more early than CPF-JDGSVD.
On the other hand, CPF-HJDGSVD converges
steadily in the first twenty-two outer iterations, and it then drops
sharply and achieves the convergence criterion in the next two outer iterations.
As the results indicate, CPF-HJDGSVD uses seven and nine fewer outer iterations
than CPF-JDGSVD and IF-HJDGSVD, respectively.

Obviously, for this problem, both CPF-HJDGSVD and IF-HJDGSVD
performs better than CPF-JDGSVD.
Of the two harmonic algorithms, CPF-HJDGSVD is favorable for its
faster overall convergence.

\begin{exper}\label{exp3}
We compute ten GSVD components of the matrix pairs
$(A_1,B_1)=(\mathrm{rajat03},T)$,
$(A_2,B_2)=(\mathrm{lp\_bnl2}^T,T)$,
$(A_3,B_3)=(\mathrm{Hamrle2},T)$ and
$(A_4,B_4)=(\mathrm{jendrec1}^T,T)$
with the generalized singular values closest to the targets
$\tau_1=50$, $\tau_2=17$, $\tau_3=8$
and $\tau_4=6.3$, respectively.
The desired generalized singular values of $(A_1,B_1)$ and $(A_2,B_2)$
are the largest ones, which are fairly isolated one another,
and those of $(A_3,B_3)$ and $(A_4,B_4)$ are highly clustered interior ones.
In the expansion phase of the three algorithms, we use {\sf minres} without
preconditioning to solve all the correction equations.
\end{exper}

\begin{figure}[tbhp]
  \centering
  \includegraphics[width=0.90\textwidth]{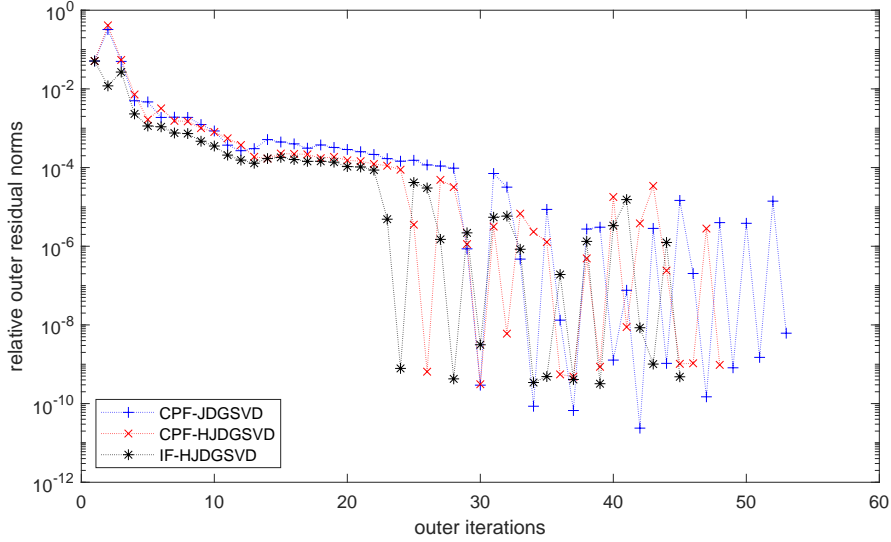}
  \caption{Computing ten GSVD components of
  $(A,B)=(\mathrm{rajat03},T)$ with $\tau=50$.}\label{fig3}
\end{figure}

Figures~\ref{fig3}--\ref{fig4} depict the convergence curves of the three
JDGSVD algorithms for computing the ten desired GSVD components of $(A_1,B_1)$ and $(A_2,B_2)$,
and Table~\ref{table2} displays the results on the four test problems.

\begin{figure}[tbhp]
  \centering
  \includegraphics[width=0.90\textwidth]{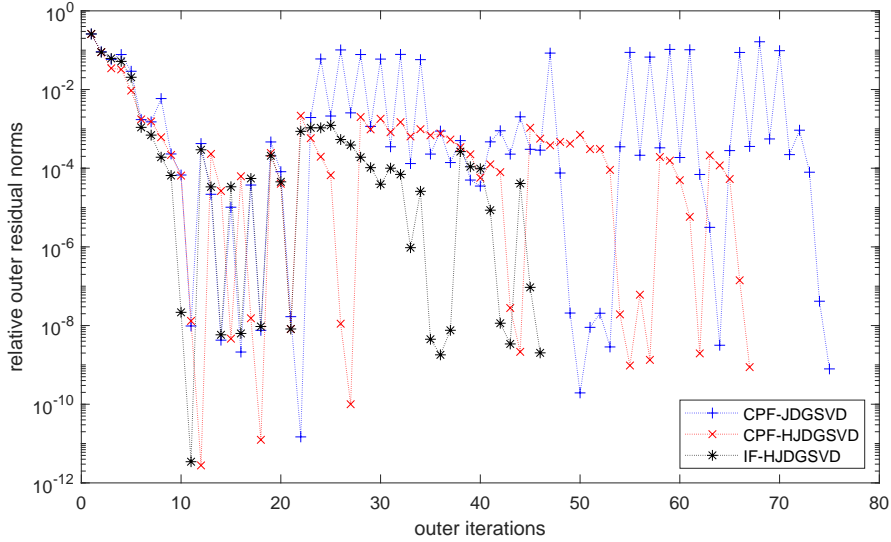}
  \caption{Computing ten GSVD components of
  $(A,B)=(\mathrm{lp\_bnl2}^T,T)$ with $\tau=17$.}\label{fig4}
\end{figure}

\begin{table}[tbhp]
	\caption{Results on the test matrix pairs in Example~\ref{exp3}}\label{table2}
	\begin{center}
		\begin{tabular}{crccc} \toprule
			$A$
			&\ \ \ Algorithm \  &$\ \ \ I_{out}\ \ \ $&\ \ \ $I_{in}$\ \ \ & \ \ \ $T_{cpu}$\ \ \ \\\midrule
			\multirow{3}{*}{rajat03}
			&CPF\ \ \ \ &53&14695&3.70\\
			&CPFH\ \ \ \ &48&13082&3.48\\
			&IFH\ \ \ \ &45&13207&3.55\\[0.2em] 
			
			\multirow{3}{*}{$\mathrm{lp\_bnl2}^T$}
			&CPF\ \ \ \ &75&4971&0.49\\
			&CPFH\ \ \ \ &67&4609&0.48\\
			&IFH\ \ \ \ &46&4477&0.42\\[0.2em]
			
			\multirow{3}{*}{Hamrle2}
			&CPF\ \ \ \ &100&17210&3.50\\
			&CPFH\ \ \ \ &113&16330&3.60\\
			&IFH\ \ \ \ &72&17214&3.69\\[0.2em] 
			
			\multirow{3}{*}{\ $\mathrm{jendrec1}^T$\ }
			&CPF\ \ \ \ &102&13382&2.06\\
			&CPFH\ \ \ \ &62&9469&1.53\\
			&IFH\ \ \ \ &42&8848&1.31 \\
			\bottomrule
		\end{tabular}
	\end{center}
\end{table}

For $(A_1,B_1)$ and $(A_2,B_2)$, we can observe from the figures
and Table~\ref{table2} that, regarding the outer convergence,
CPF-HJDGSVD and especially IF-HJDGSVD outperform CPF-JDGSVD
as they use a little bit fewer and substantially fewer outer
iterations than the latter for $(A_1,B_1)$ and $(A_2,B_2)$, respectively.
Specifically, for $(A_2,B_2)$, we see from Figure~\ref{fig4} that the two
harmonic algorithms CPF-HJDGSVD and IF-HJDGSVD have much smoother
and faster outer convergence. We must remind the reader
that, for $\ell=10$, each JDGSVD algorithm has ten convergence stages,
which denote the one by one convergence processes of the desired
ten GSVD components. In the meantime, we also see from Table~\ref{table2} that,
regarding the overall efficiency, CPF-HJDGSVD and IF-HJDGSVD
outperform CPF-JDGSVD in terms of total inner iterations and
total CPU time.


For $(A_3,B_3)$, since the desired generalized singular
values are highly clustered, the corresponding left and right generalized
singular vectors are ill conditioned.
As a result, it may be hard to compute the desired GSVD components using
the standard and harmonic JDGSVD algorithms \cite{jia2004some}.
We observe quite irregular convergence behavior and sharp
oscillations of CPF-JDGSVD and CPF-HJDGSVD, while IF-HJDGSVD
converges much more smoothly
and uses significantly fewer outer iterations, compared with
CPF-JDGSVD and CPF-HJDGSVD, as shown in Table~\ref{table2}.
Therefore, IF-HJDGSVD is preferable for this problem.

For the matrix pair $(A_4,B_4)$, the three JDGSVD algorithms
succeed in computing all the desired GSVD components.
Among them, CPF-HJDGSVD outperforms CPF-JDGSVD considerably in
terms of outer iterations and overall efficiency, and
IF-HJDGSVD is slightly better than CPF-HJDGSVD as it uses
quite fewer outer iterations and slightly fewer inner iterations
and less CPU time than the latter one.
Clearly, both CPF-HJDGSVD and IF-HJDGSVD are suitable for this
problem and  IF-HJDGSVD is favorable due to the faster outer convergence.

In summary, for the four test problems, IF-HJDGSVD performs best,
CPF-HJDGSVD is the second, and both of them are considerably better
than CPF-JDGSVD.

\begin{exper}\label{exp4}
We compute the ten GSVD components of
$(A,B)\!=\!(\mathrm{grid2},L_1)$ with the desired
generalized singular values closest to the target $\tau=4e+2$.
The desired generalized singular values are the largest
ones and well separated one another.
\end{exper}

\begin{figure}[tbhp]
  \centering
  \includegraphics[width=0.90\textwidth]{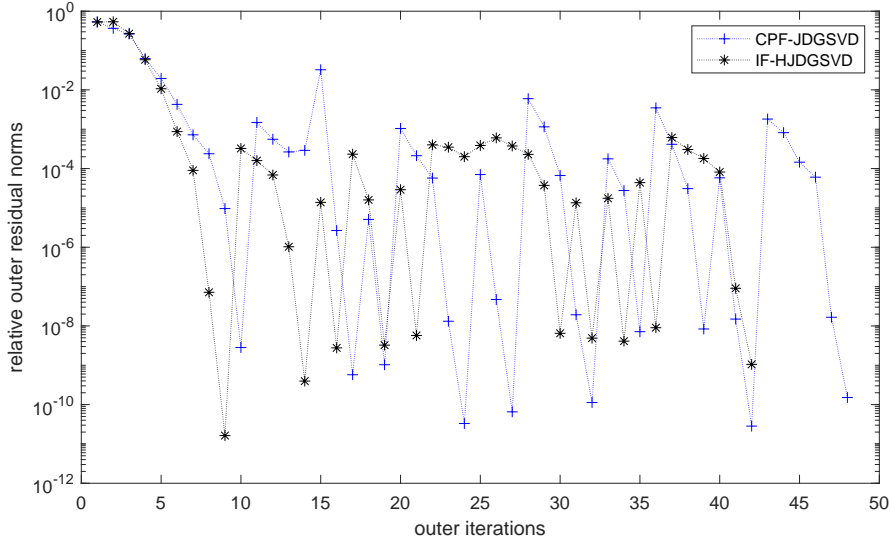}
  \caption{Computing ten  GSVD components of
  $(A,B)=(\mathrm{grid2},L_1)$ with $\tau=4e+2$.}\label{fig5}
\end{figure}

For the matrix pairs $(A,B)$ with $B$ rank deficient, CPF-HJDGSVD cannot be applied.
We only use CPF-JDGSVD and IF-HJDGSVD to compute the desired GSVD components of $(A,B)$ and report the results obtained. The outer iterations, inner iterations and CPU time used by CPF-JDGSVD are 48, 71317 and 6.6 seconds, respectively, and those used by IF-HJDGSVD are 42, 67463 and 6.3 seconds, respectively. In Figure~\ref{fig5}, we draw the outer convergence curves of these two algorithms.

As can be seen from Figure~\ref{fig5} and the data listed above,
IF-HJDGSVD outperforms CPF-JDGSVD
in terms of the outer iterations, the overall efficiency and smooth\deleted{er}
convergence behavior.


\begin{exper}\label{exp5}
We compute the ten GSVD components of the matrix pairs
$(A_1,B_1)\\=(\mathrm{dw1024},L_1)$,
$(A_2,B_2)\!=(\mathrm{\mathrm{r05}^T},L_1)$,
$(A_3,B_3)\!=(\mathrm{\mathrm{p05}^T},L_1)$,
$(A_4,B_4)=(\mathrm{bibd\_81\_2},\\L_2)$,
$(A_5,B_5)=(\mathrm{benzene},L_2)$ and
$(A_6,B_6)=(\mathrm{blckhole},L_2)$
with the generalized singular values closest to the targets
$\tau_1=30$,
$\tau_2=40$,
$\tau_3=300$,
$\tau_4=150$,
$\tau_5=3$ and
$\tau_6=400$,
respectively.
All the desired generalized singular values are interior ones and are
fairly clustered, except for
$(A_1,B_1)$, whose desired generalized singular values are well
separated one another.
\end{exper}

\begin{table}[tbhp]
	\caption{Results on test matrix pairs in Example~\ref{exp5}}\label{table4}
	\begin{center}
		\begin{tabular}{ccccccc} \toprule
			\multirow{2}{*}{$A$}
			&\multicolumn{3}{c}{CPF-JDGSVD}&\multicolumn{3}{c}{IF-HJDGSVD}\\
			\cmidrule(lr){2-4}
			\cmidrule(lr){5-7}
			&\ \ $I_{out}$\ \ &$I_{in}$&\ \ \ $T_{cpu}$\ \ \ &\ \ $I_{out}$\ \ &$I_{in}$&\ \ \ $T_{cpu}$\ \ \ \\ \midrule
			
			{dw1024}&62&61063&3.61&47&49560&2.86\\  
			
			{$\mathrm{r05}^T$}&73&58292&14.9&44&56257&15.0\\  
			
			{$\mathrm{p05}^T$}&48&111177&24.7&40&96114&22.5\\  
			
			{\ bibd\_81\_2\ }&166&484601&39.9&112&314748&27.3\\  
			
			{benzene}&65&154109&88.9&41&109394&61.3\\ 
			
			{blckhole}&180&356204&23.5&128&242227&16.1\\  
			
			\bottomrule
		\end{tabular}
	\end{center}
\end{table}

Table~\ref{table4} displays all the results obtained.
As is observed from them, for the matrix pairs $(A_1,B_1)$, $(A_3,B_3)$, $(A_4,B_4)$,
$(A_5,B_5)$ and $(A_6,B_6)$ with the given targets,
IF-HJDGSVD uses fewer outer and inner iterations and less CPU time to converge
than CPF-JDGSVD, and it outperforms CPF-JDGSVD either slightly or significantly.
For $(A_2,B_2)$, however, IF-HJDGSVD uses much fewer outer iterations but comparable inner iterations and CPU time to compute all the desired GSVD components,
compared with CPF-JDGSVD. In terms of a smoother and faster outer convergence, IF-HJDGSVD outperforms CPF-JDGSVD for this problem; almost the same
overall efficiency, i.e., $I_{inner}$ and $T_{cpu}$, is due to
approximate solutions of correction equations using the MINRES method, whose
convergence is complicated and depends on several factors, especially
when a linear system is highly indefinite.

\vspace{0.5em}

Summarizing all the numerical experiments, we conclude that
(\romannumeral1) for the computation of large GSVD components,
CPF-HJDGSVD and IF-HJDGSVD generally suit better than CPF-JDGSVD,
(\romannumeral2) for the computation of interior GSVD components,
CPF-HJDGSVD and IF-HJDGSVD generally outperform
CPF-JDGSVD, and, of them, IF-HJDGSVD
is often favorable due to its faster and smoother
convergence, higher overall efficiency and wider applicability,
and (\romannumeral3) for the computation of small GSVD components,
if $B$ is of full column rank, then CPF-HJDGSVD performs slightly
better than IF-HJDGSVD and both of them are preferable to CPF-JDGSVD.

\section{Conclusions}\label{sec:7}

In this paper, we have proposed two harmonic extraction based JDGSVD methods
CPF-HJDGSVD and IF-HJDGSVD that are more suitable for the computation of interior
GSVD components of a large matrix pair. The algorithms are
$A^TA$ and $B^TB$ free and their inversions free, respectively.
To be practical,
we have developed their thick-restart algorithms with efficient
deflation and purgation to compute more than one
GSVD components of $(A,B)$ with a given target $\tau$. We have
detailed a number of key issues on subtle efficient implementations.

We have made numerical experiments on a number of problems,
illustrating that both IF-HJDGSVD and CPF-HJDGSVD outperform
CPF-JDGSVD and can be much better than CPF-JDGSVD, especially
for the computation of interior GSVD components. Furthermore, we have observed
that IF-HJDGSVD is generally more robust and reliable than CPF-HJDGSVD
and, therefore, is preferable but CPF-HJDGSVD is a
better option when small GSVD components are required and $B$ has
full column rank.

However, as we have observed, IF-HJDGSVD and CPF-HJDGSVD,
though better than CPF-JDGSVD, may perform badly for some test problems,
and they may exhibit irregular convergence behavior. This
is most probably due to the intrinsic possible irregular convergence
and even non-convergence of a harmonic extraction approach,
which states that harmonic Ritz vectors may converge irregularly
and even fail to converge even though the distances between desired
eigenvectors or, equivalently, (generalized) singular vectors
and searching subspaces tend to zero; see \cite{jia2005}.
Such potential drawback has
severe effects on effective expansions of searching subspaces
and strongly affects the convergence of the resulting harmonic extraction
based algorithms. To better solve the GSVD problem in this paper,
a refined or refined harmonic
extraction based JDGSVD type algorithm should be
appealing. This will constitute our future work.

\section*{Statements and Declarations}

This work was supported by the National Science Foundation of
China (No. 12171273). The two authors declare that they have no
financial interests, and the two authors read and approved the final manuscript.
The datasets generated during and/or analysed during the current study are available from the corresponding author on reasonable request.

\bibliographystyle{siamplain}

\end{document}